\documentclass[11pt,a4paper]{article}

\usepackage{amsmath, amssymb, amsthm}
\usepackage{authblk} 
\usepackage{fancyhdr}
\usepackage{graphicx}
\usepackage{times}
\usepackage{enumitem}

\usepackage{caption}
\usepackage{subcaption}

\usepackage[backend=biber,style=ieee]{biblatex}
\usepackage{hyperref}
\hypersetup{
    colorlinks=true,
    citecolor=black,
    anchorcolor=black,
    linkcolor=black,
    filecolor=black,      
    urlcolor=blue,
    }
\usepackage{tikz}
\usepackage[left=2.5cm,right=2.5cm,top=3.8cm,bottom=3.0cm]{geometry}
\usepackage[nameinlink]{cleveref}

\usetikzlibrary{backgrounds}

\newtheoremstyle{tightplain}
  {\topsep}   
  {3pt}       
  {\itshape}  
  {}          
  {\bfseries} 
  {.}         
  {.5em}      
  {}          

\theoremstyle{tightplain}
\newtheorem{theorem}{Theorem}[section]
\newtheorem{lemma}[theorem]{Lemma}  

\newtheorem{definition}[theorem]{Definition}

\makeatletter
\renewcommand{\section}
{\@startsection{section}{1}{0mm}{-\baselineskip}{1mm}{\large \bf}}
\renewcommand{\subsection}
{\@startsection{subsection}{1}{0mm}{-\baselineskip}{1mm}{\normalsize \bf}}
\renewcommand{\subsubsection}
{\@startsection{subsubsection}{1}{0mm}{-\baselineskip}{1mm}{\normalsize \it}}
\makeatother

\newlength{\pagenumbershift}
\def\thepage{\textit{\arabic{page}}}

\usepackage{setspace}

\usepackage{titlesec}

\titleformat*{\section}{\large\bfseries\parskip=-5pt}
\titleformat*{\subsection}{\normalsize\bfseries}

\usepackage{caption}
\titlespacing*{\section}{0pt}{15.5pt}{-4pt}
\titlespacing*{\subsection}{0pt}{5pt}{-4pt}
\titlespacing*{\subsubsection}{0pt}{5pt}{-4pt}

\fancypagestyle{copyrightfooter}{%
  \fancyhf{}
  
  \fancyfoot[L]{
  \small \em \vspace{-35pt} 
  \rule{\textwidth}{0.35pt}\vspace{-1pt} 
  Copyright \textcopyright\,2026 by 
  $<$Cameron Millar, Bernd Schulze, Louis Theran$>$  
  \\
  Published by the International Association for Shell and Spatial Structures (IASS) with permission.
  }
}

\begin{document}
\setlength{\baselineskip}{13pt}


\begin{center}
\vspace{-15pt}
{\bf \fontsize{16}{19}\selectfont 
On weavings, grillages, tensegrities, and frameworks
} 

\vspace{0.2cm}

Cameron MILLAR*, Bernd SCHULZE$^{\,\mathrm a}$, Louis THERAN$^{\,\mathrm b}$
\\

\vspace{10pt}

{
\small
*Skidmore, Owings \& Merrill \\
The Broadgate Tower, 20 Primrose Street, London, UK, EC2A 2EW\\
cameron.millar@som.com 
}

\vspace{5pt}

{
\small
$^{\,\mathrm a}$
School of Mathematical Sciences, Lancaster University, UK
\\ 
$^{\,\mathrm b}$
School of Mathematics and Statistics, University of St Andrews, UK
}

\vspace{-5pt}

\end{center}
\noindent
{\bf \large Abstract}\\[2pt]
We investigate the  stability of discrete structures comprising of woven 
elastic beams in a plane, exposing a connection between admissible 
over / under patterns and the first-order and static rigidity of an associated 
tensegrity that is related through a natural polarity transformation.  
The relationship between the Airy and Whiteley stress functions for the tensegrity and weaving structures is explored. 
Our results lead to an efficient method for finding such an over / under pattern.
The method is illustrated through a worked example modelled on the 
complete bipartite graph $K_{4,4}$.

\begin{spacing}{1}
\vspace{5pt}\small
\end{spacing}

\vspace{3pt}


\section{Introduction}
Tarnai \cite{Tarnai} and Whiteley \cite{WW89, WW91} introduced the concept of grillages and weavings along with their polar dual counterpart, frameworks and tensegrities. 
Grillages and weavings are structures made of elastic beams lying in a plane that resist bending. The beams cross through each other, but do not slide past each other.  
The work of Tarnai and Whiteley exposed, via a classical projective polarity, a deep connection between grillages and weavings on one side and frameworks and tensegrities on the other.

The aim of this paper is to provide an account of the Tarnai--Whiteley polarity in modern 
language, and to provide a simple design methodology for designing stable over / under 
patterns for weavings. To do this, we first review the Tarnai and Whiteley polarity from a 
simpler affine perspective than the projective geometry approach taken previously \cite{Tarnai,WW89, WW91},
drawing an explicit connection between the virtual work pairing induced by a framework and the 
space of admissible liftings of a weaving.  The dual cone to the liftings of a weaving, which we call the 
\emph{Whiteley stress cone}, is naturally linearly isomorphic to the proper stress cone
(or \emph{Airy stress cone}) of the polar tensegrity.  
The conditions for static equilibrium of the beams in a weaving is explored, in turn establishing the duality to the static equilibrium conditions of the associated tensegrity. 
Furthermore, the relationship between the stresses in the structures and their respective stress cones is explored. 
In turn, this leads to a count akin to the Maxwell-Calladine Index Theorem. 

This mathematics is leveraged in a design tool which allows the engineer to determine the over / under pattern to transform a grillage with a fixed geometry into a stable weaving. 
This is done by considering the dual tensegrity framework, with cables and struts relating to the over / under weaving pattern (essentially reading the sign of the equilibrium stresses in the tensegrity). 
We illustrate the method with a worked example based on the complete bipartite graph $K_{4,4}$.



\section{Structural engineering systems}

\subsection{Tensegrity and the Airy Stress Function}

A framework is made up of discrete bars connected at nodes. The members of a framework are assumed to be 
axially rigid and connected with pins at nodes, like in a truss. 
We define a \textit{bar-joint} framework as a framework in which every member is a bar 
(ie. it can be loaded in both tension and compression). 
We define a \textit{tensegrity} as a framework in which every member is a cable or strut 
(ie. can be loaded in only tension or compression respectively). 
Tensegrity structures often posses one or more mechanisms which are stabilised and stiffened by a state of self-stress (known as \emph{pre-stress stability}). 
The reader is referred to Connelly and Guest \cite{CG22} for a detailed description of tensegrity structures. 

As developed by Maxwell and Airy , the states of self-stress within a framework can be described through the Airy stress function. 
The discrete Airy stress function is a closed plane-faced polyhedron whose projection is the 2D framework (the form diagram). The signed force in a member is given by the signed change in slope across the corresponding edge in the Airy stress function. 
Similarly, the mechanisms in the framework can also be found through an incremental Airy stress function over the dual force diagram \cite{Baker}. 
This can be very powerful in the design of structures as it allows the engineer to visualise and control the forces within a structure. 
In \emph{graphic statics}, the \emph{force diagram} construction comes directly from the Airy stress function over the \emph{form diagram} in Maxwell's construction (via a duality) \cite{Baker}. 

\subsection{Grillages and Weavings}

A grillage is a collection of beams which lie in the plane. Beams relate to bending or flexure in contrast to the axial bars considered in frameworks. 
In a grillage, each joint between beams does not transfer moments, but rather allows crossing beams to exert a vertical force onto the other. 

A weaving is a collection of beams which roughly lie in the plane, but are \textit{actively bent} to produce an `over / under' pattern. At each joint, it is only possible for the `under' beam to exert an upwards force on the `over' beam, and vice versa -- this can be considered a compression only connection between the beams. 
It is assumed that there is no friction nor in-plane force transfer between the beams. 
This pattern can be shown through linework through the adoption of broken and continuous lines, as in Tarnai \cite{Tarnai}. 

\begin{figure}[h]
    \centering
    \includegraphics[width=0.3\linewidth]{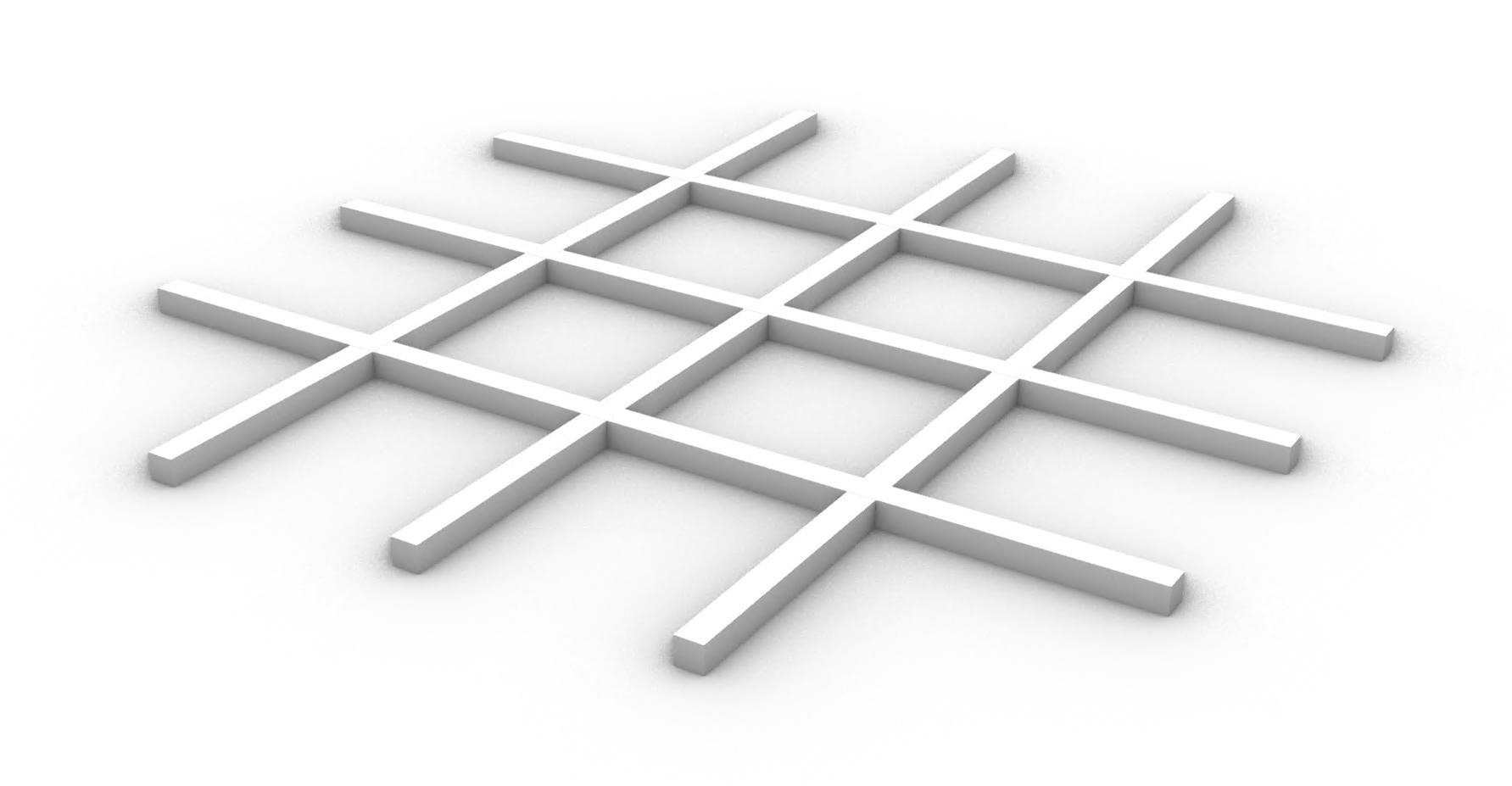}
    \hspace{0.2cm}
    \includegraphics[width=0.3\linewidth]{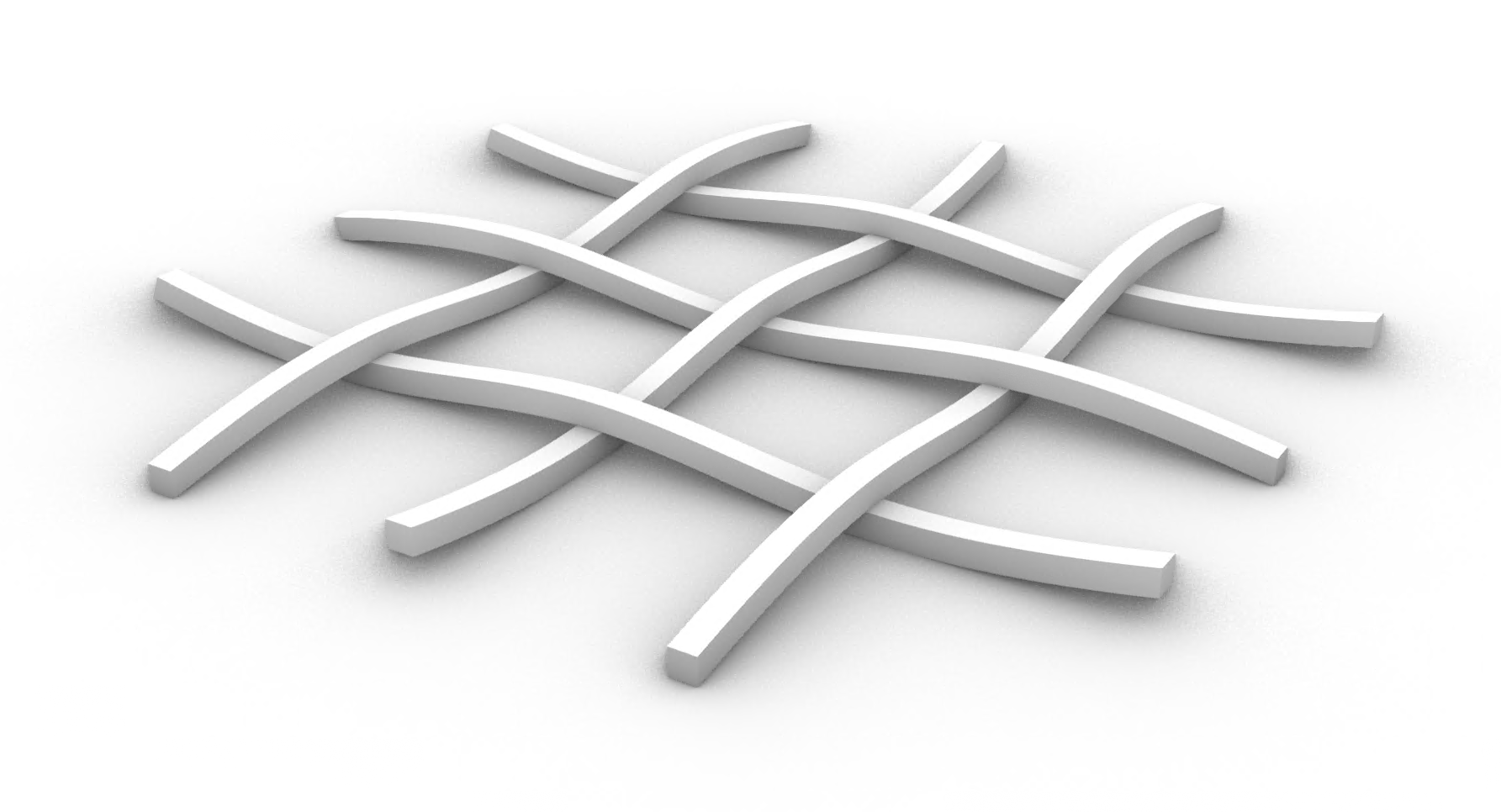}
    \hspace{0.2cm}
    \includegraphics[width=0.22\linewidth]{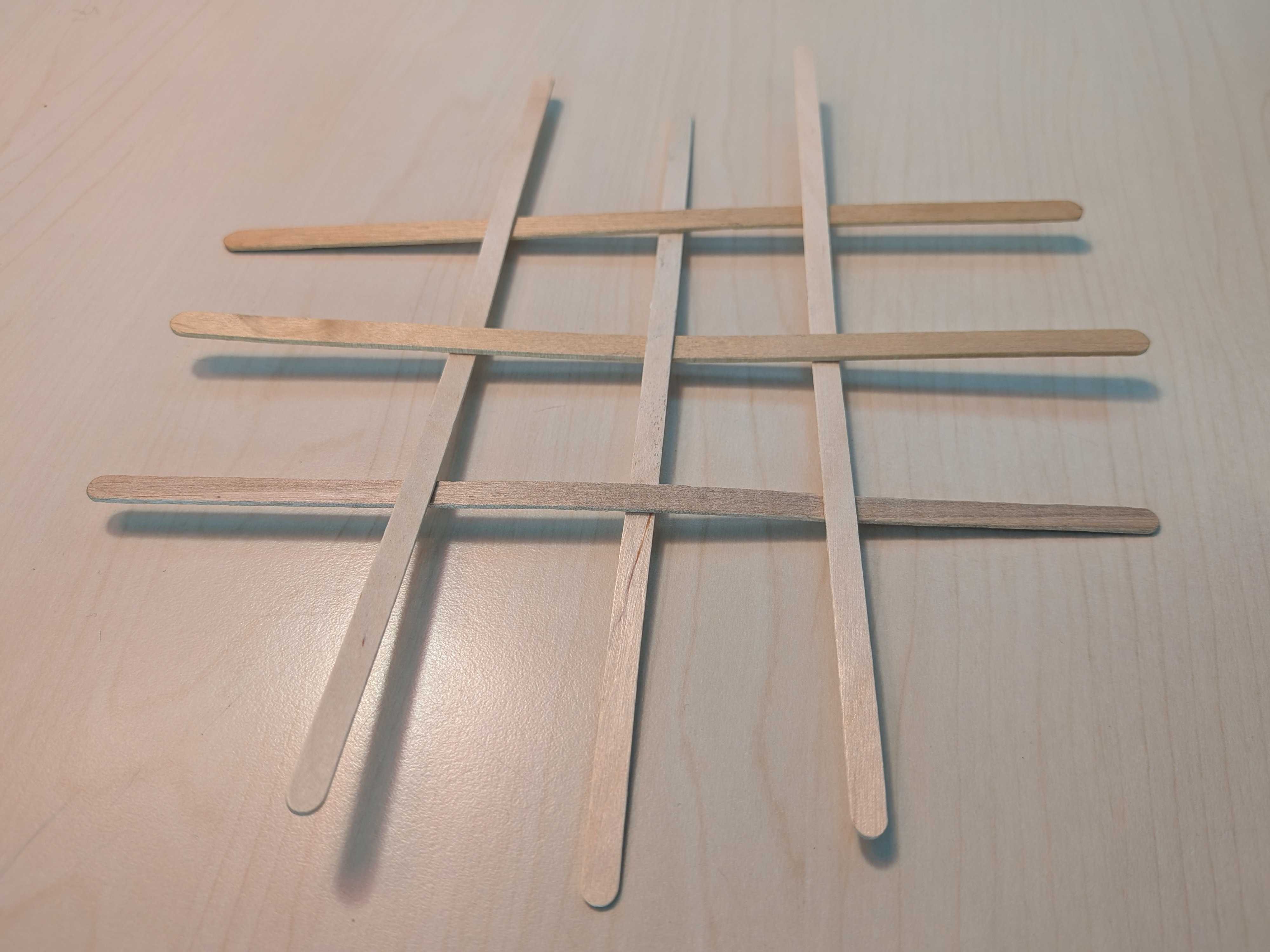}
    \caption{A grillage, a weaving, and the associated `popsicle bomb' made from coffee stirrers. }
    \label{fig:weaving}
\end{figure}

A commonplace example of a weaving is a `popsicle bomb' which stores energy through the principal of 
active bending within the beams. Structures like this have interesting and useful properties.
Firstly, weavings increase the structural depth and therefore have increased out-of-plane stiffness in comparison 
to grillages \cite{SOM}. Secondly, by ensuring connections are entirely compressive, a much simpler node detail can be 
employed. In the particular case of timber beams, it may be feasible for this node to be fabricated without steel / 
mechanical fasteners and therefore the structure could also be disassembled and re-assembled. 
As such, a design methodology for weavings could be very powerful. 

\section{Mathematical preliminaries }
We briefly review some background material that can be found in more detail in Connelly and Guest \cite{CG22}.
Let $G = (V,E)$ be a graph with vertex set $V = \{1,\dots,n\}$ and edge set $E$, where $|E|=m$.  
A \emph{configuration} $p = (p_1,\dots,p_n)$ assigns a point $p_i \in \mathbb{R}^d$ to each vertex $i \in V$.  
The pair $(G,p)$ is called a \emph{bar-and-joint framework} (or simply \emph{framework}) in $\mathbb{R}^d$
(in this paper $d=2$).  
Each edge $ij \in E$ represents a rigid bar connecting vertices $i$ and $j$, which constrains the 
distance between $p_i$ and $p_j$.

A $d$-dimensional \emph{tensegrity} $(G,p)$ is a framework together with an assignment 
$\varepsilon_{ij}=\pm 1$ for each edge, 
where $\varepsilon_{ij} = +1$ corresponds to a strut (which can expand but not contract) 
and $\varepsilon_{ij} = -1$ corresponds to a cable (which can contract but not expand).

\subsection{Infinitesimal rigidity of tensegrities}

An \emph{infinitesimal flex}\footnote{In engineering literature, this is often referred to as a \emph{mechanism}.}, $p'$, of a tensegrity $(G,p)$ is a solution to the equations 
\begin{equation}\label{eq: tensegrity inf flex}
    \varepsilon_{ij}\langle p_i - p_j,p'_i - p'_j\rangle \ge 0,
\end{equation}
An infinitesimal flex is a \emph{bar flex} if the l.h.s. is zero
(i.e., the inequality holds with equality). 

\begin{definition} The 
\emph{flex cone} of $(G,p)$ is the polyhedral cone of the infinitesimal 
flexes of $(G,p)$. 
\end{definition}
The flex cone always contains a linear subspace of 
bar flexes, and we say that $(G,p)$ is \emph{stretched}\footnote{This is a 
non-standard term. }
if the flex cone is exactly equal to the bar flexes and \emph{infinitesimally 
rigid} if it is stretched and all the bar flexes are trivial. Note that, 
for a $2$D tensegrity, there is always a $3$-dimensional space of trivial bar 
flexes (corresponding to rigid body motions), 
provided the tensegrity does not lie on a line.

Let us define a virtual work 
pairing $(-,-)$ between $(\mathbb{R}^2)^n$ and 
$(\mathbb{R}^m)^*$ by 
\begin{equation}\label{eq: tensegrity pairing}
    (\omega,p') = 
    \sum_{ij\in E}\omega_{ij}\varepsilon_{ij}\langle p_i - p_j,p'_i - p'_j\rangle.
\end{equation}
Since we are considering kinematics, we read the pairing 
as follows.  We identify $\mathbb{R}^{2n}$ with 
the space of infinitesimal flexes.  The rigidity matrix\footnote{This is associated with the compatibility matrix in common structural engineering literature, and is also the transpose of the equilibrium matrix. The difference is that $R(p) = \operatorname{diag}(\|p_i - p_j\|)C(p)$. }, $R(p)$, 
is the $m\times 2n$ matrix of the linear transformation
\[
    p'\mapsto \left(
    \langle p_i - p_j, p'_i - p'_j\rangle
    \right)_{ij\in E(G)}
\]
that transforms flexes to infinitesimal edge extensions, so
in the standard basis, we have 
\[
    (\omega,p') = \omega^\top 
    \operatorname{diag}(\varepsilon_{ij})R(p)p'.
\]
The stress\footnote{Stress is a term from rigidity theory; each coefficient $\omega_{ij}$ of a stress $\omega$ is called a bar axial force in common structural engineering literature. }, $\omega$, which corresponds to forces in the members, 
acts on the extensions, so it is an element of the 
dual space $(\mathbb{R}^m)^*$.

\begin{definition}
Under virtual work pairing above, 
the dual of the flex cone is the \emph{proper stress cone}
\[
    \{\omega \ge 0 : (\omega,p') = 0 \text{ for all flexes $p'$}\}.
\]
\end{definition}
Infinitesimal rigidity of a tensegrity is characterized by infinitesimal 
rigidity of an underlying bar framework and a stress condition.
\begin{lemma}\label{lem: tensegrity tight}
A tensegrity $(G,p)$ is stretched if and only if there is a positive 
vector in the proper stress cone.
\end{lemma}
Assuming that there is a positive $\omega$ in the stress 
cone, the lemma says that the only tensegrity flexes are 
bar flexes, and so we obtain the following:
\begin{theorem}\label{thm: tensegrity inf rigidity}
A tensegrity $(G,p)$ is infinitesimally rigid 
if and only if $(G,p)$ is 
infinitesimally rigid as a bar framework and there is a positive 
vector in the proper stress cone.
\end{theorem}

\subsection{Static rigidity of tensegrities}
Now we review how kinematics and statics of tensegrities are 
dual theories.  To this end, we assume that an external 
load, $f$, is applied at the nodes of the tensegrity.  By Newton's 
third law, equilibrium requires no net translational force and no 
net torque\footnote{Torque has to be computed as a rotational
moment around a reference point, but once there is no net 
force, the choice doesn't matter --- the origin has been chosen for simplicity. }.  
These conditions are encoded by the equations
\begin{equation}\label{eq: equilibrium load}
     0 = \sum_{i\in V(G)} f_i\qquad \text{and}\qquad 0 = 
    \sum_{i\in V(G)} \langle \rho(p_i),f_i\rangle
\end{equation}
where $\rho$ is a counter-clockwise rotation through an angle of 
$\pi/2$.  Such a load $f$ is called an \emph{equilibrium 
load}.
We then say that a non-negative\footnote{For bar frameworks, we 
drop the non-negative requirement.}
stress $\omega$ resolves $f$ if, for each vertex $i$, 
\[
    -f_i = \sum_{j\sim i}\omega_{ij}\varepsilon_{ij}(p_i - p_j),
\]
which is the equation 
\(
    -f = R^*(p)\omega.
\)
A tensegrity is called \emph{statically rigid} if every equilibrium load 
is resolved by some non-negative stress.
Since $-f$ failing to be in the image of the non-negative 
stresses is witnessed by a $v\in \mathbb{R}^{2n}$
such that $\langle -f,v\rangle\neq 0$ but $(\omega,v) = 0$
for all $\omega\ge 0$, we get 
the standard result:
\begin{theorem}\label{thm: tensegrity static kinematic duality}
A tensegrity $(G,p)$ is infinitesimally rigid if and only if 
it is statically rigid.
\end{theorem}

\section{Polar duality of tensegrities and weavings}
Here, we revisit duality results of Whiteley and Tarnai \cite{Tarnai,WW89, WW91}.
Just as the primitive objects of a tensegrity are its joints and 
elements, for a weaving the corresponding concepts are the 
lines modeling its beams and the crossing points.
This section initially develops the polar dual relationship between tensegrities and weavings from both a kinematic and static perspective. It then discusses the index theorem through the lens of previously introduced terms and the duality between tensegrities and weavings. 
Finally, it develops the duality between the Airy stress function and the `Whiteley stress function'.


\subsection{Geometric model of a weaving}
Let us describe a line $L$ in $\mathbb{R}^2$ by an affine linear equation%
\footnote{This form requires that $L$ not pass through the origin, which is easy to 
arrange in a finite system.}%
: $L = \{x \in \mathbb{R}^2 : \langle x,p\rangle = 1\}$, for a fixed $p\in \mathbb{R}^2$.
The crossing point $q$ of two lines $L$ and $L'$ that are not parallel 
is
\(
    q = s_{ij}\rho(p - p'), 
\)
where $s_{ij} = 1/\det [p\, p']$.
The formal definition of a weaving is as follows.
\begin{definition}[Weaving]
A \emph{weaving} is a set of $n$ lines $\{L_1,\dots,L_n\}$,
where $L_i$ is defined by a point $p_i\in \mathbb{R}^2$,
a graph $G=(V,E)$ with vertex set $\{1, \ldots, n\}$, and 
for each edge $ij$ of $G$ with $i < j$, an assignment 
$\varepsilon_{ij}\in\{\pm 1\}$ which records the over / under relation at each
crossing.
The convention is
that $\varepsilon_{ij}=+1$ means $L_i$ passes over $L_j$ at $q_{ij}$,
while $\varepsilon_{ij}=-1$ means that $L_i$ passes under $L_j$.
\end{definition}

\subsection{Liftings of a weaving}
Following Whiteley \cite{WW89, WW91}, a weaving starts in the $xy$-plane.  Its allowed 
`motions' lift the lines into $\mathbb{R}^3$, subject to constraints, which we 
now explain. A fixed vector $v\in \mathbb{R}^2$ defines a 
linear map $\mathbb{R}^2\to \mathbb{R}$ by $w\mapsto \langle w, v\rangle$.
A linear height function on a line $L$ is the restriction of such a 
map to $L$.  The lifting of a line $L$ by a linear height function 
$v$ is the line $L_v = \{\bar x + \langle x,v\rangle e_3 : x\in L\}$, 
where $\bar x$ extends the vector $x\in \mathbb{R}^2$ by a zero 
coordinate and $e_3 = (0,0,1)^\top$.

If $L$ and $L'$ are lines
crossing at point $q$ are lifted with height functions $v$ and 
$v'$, then the difference in heights of $L_v$ and $L'_{v'}$ over 
$q$ is given by $\langle q,v - v'\rangle$.  The 
definition of a  lifting of a weaving says that the over / 
under constraints are satisfied.
\begin{definition}[Lifting] A \emph{lifting} of a weaving is an assignment of a linear height 
function $v_i$ to each $L_i$ satisfying 
\begin{equation}\label{eq: lifting ineuality}
        \varepsilon_{ij}\langle q_{ij}, v_i - v_j\rangle\ge 0.
\end{equation}
The \emph{lifting cone} is the set of solutions to the lifting inequalities, with the weaving fixed
and the $v_i$ variable.
\end{definition}
The lifting cone is a polyhedral cone in $\mathbb{R}^{2n}$.  A lifting $v$ is called \emph{tight} if the lifting 
inequalities are satisfied with equality.
There is always a $3$-dimensional linear subspace of tight liftings, corresponding
to the freedom of choosing an affine plane in $\mathbb{R}^3$, which we can describe 
explicitly as follows. Setting $v_i = v_j$ for all $i$ and $j$ produces a tight 
lifting.  
We have $q_{ij} = s_{ij}\rho(p_i - p_j)$.  So setting $v_i = p_i$ leads to 
\[
    \langle q_{ij}, v_i - v_j\rangle = \langle s_{ij}\rho(p_i - p_j), p_i - p_j\rangle = 0,
\]
which shows we have a tight lifting. We call any lifting in the linear span of these liftings \emph{trivial}.  
\begin{definition}
A weaving is called \emph{tight} if every point in the lifting cone corresponds to a lifting 
where the inequalities hold with equality, i.e. a tight lifting. A weaving is called \emph{flat} if the lifting cone 
is equal to the linear space of trivial liftings. 
\end{definition}
We  mention that there are tight weavings that are not flat, which is why we introduce 
both concepts.
Another useful notion is that of a \emph{flat grillage}, which is a weaving where all the 
liftings with all inequalities tight are trivial.  

To understand flat weavings, we introduce the Whiteley stress, which is 
formally similar to an equilibrium stress in a tensegrity.
We induce a virtual work pairing $(-,-)$ between $(\mathbb{R}^m)^*$ and $\mathbb{R}^{2n}$ by 
\begin{equation}\label{eq: weaving paring}
    (\omega,v) = \sum_{ij\in E} \omega_{ij}\varepsilon_{ij}\langle q_{ij}, v_i - v_j\rangle,
\end{equation}
which we then use to define a stress cone.
In the standard basis, the matrix 
of this virtual work pairing is 
\[
    \operatorname{diag}(\varepsilon_{ij}s_{ij})R(p)(I_n\otimes \rho^{\top}),
\]
where $R(p)$ is the rigidity matrix of the framework $(G,p)$, 
the diagonal matrix on the left is the stress scaling in the 
isomorphism from Whiteley stresses to Airy stresses in 
\Cref{thm: whiteley stress polarity} below and the Kronecker product on the 
right is the geometric polarity acting on the domain.  This is the 
equivalent of the rigidity matrix for a weaving and can relate the applied loads to the internal forces.

\begin{definition}
Using the pairing above, we define the \emph{Whiteley stress cone} of the weaving to be the dual of the lifting cone:
\[
    \{\omega \ge 0: \text{$(\omega,v) = 0$ for all liftings $v$}\}.
\]
We call the elements of the Whiteley stress cone \
\emph{proper stresses}.
\end{definition}
If $\omega$ is a Whiteley stress, then $\omega_{ij}/|s_{ij}|$ is the force between 
beams $L_i$ and $L_j$ at their crossing point. 
Thus, the Whiteley stress has units of Force $\times$ Length\textsuperscript{-2} ($s_{ij}$ has units of Length\textsuperscript{-2}).
Whiteley stresses stabilize weavings in a manner similar to how equilibrium stresses 
stabilize a tensegrity.
\begin{lemma}\label{lem: weaving tight}
A weaving has only tight liftings if and only if it has a positive proper stress.
\end{lemma}
\begin{proof}
First suppose that the weaving has a positive proper stress $\omega$ (this is the easy direction).  Let $v$ be a 
lifting.  By hypothesis, 
\[
    0 = (\omega,v) = \sum_{ij\in E} \omega_{ij}\varepsilon_{ij}\langle q_{ij}, v_i - v_j\rangle.
\]
Because $v$ is a lifting and $\omega_{ij} > 0$, every term 
must vanish, which shows that $v$ is tight.


Now suppose that the weaving has no positive proper stress (this is the harder direction).  Let 
$Y\subseteq \mathbb{R}^m$ be the linear subspace of 
$\{\tau : \text{$(\tau,v) = 0$ for all liftings $v$}\}$. 
Note that the proper stress cone is the intersection of $Y$ with the
non-negative orthant, so the hypothesis that there is no positive proper stress implies that 
$Y$ is disjoint from the interior of the non-negative orthant.  Since $\mathbb{R}^m = Y\oplus Y^\perp$, 
we conclude that $Y^\perp$ contains a non-zero, non-negative vector $w$.  From the definition of the pairing, $Y^\perp$ is the image of the linear map
\begin{equation*}
    u\mapsto \left(\varepsilon_{ij}\langle q_{ij},u_i - u_j\rangle\right)\qquad \mathbb{R}^{2n}\to \mathbb{R}^m,
\end{equation*}
so there is a pre-image $v$ of $w$ in $\mathbb{R}^{2n}$.  Non-negativity of $w$ implies that $v$ is a non-tight 
lifting.
\end{proof}
As a corollary we obtain Whiteley's first result.
\begin{theorem}\label{thm: whiteley flat weaving}
A weaving is flat if and only if it is a flat grillage and it has a positive proper stress.
\end{theorem}

\subsection{Kinematic correspondence with tensegrities}
Now we are ready to describe the polarity.  Let us 
fix a weaving $(G,p)$ with signs $\varepsilon$.  The 
\emph{polar tensegrity} is defined by $(G,p)$ and 
signs 
$\varepsilon'_{ij} = \operatorname{sign}(s_{ij})\varepsilon_{ij}$.
Suppose that $\omega$ is a proper stress of the weaving.  So we 
have, for each vertex $i$ 
\[
    0 = \sum_{j : ij\in E} \omega_{ij}\varepsilon_{ij}q_{ij} 
    \quad 
    \Longleftrightarrow
    \quad 
    0 = 
    \sum_{j : ij\in E} |s_{ij}|(\frac{1}{|s_{ij}|}\omega_{ij})\varepsilon'_{ij}q_{ij} = 
    \rho\left(\sum_{j : ij\in E} (\frac{1}{|s_{ij}|}\omega_{ij})\varepsilon'_{ij}(p_i - p_j)\right),
\]
Since $\rho$ is invertible, this shows that 
$\omega' = (\omega_{ij}/|s_{ij}|)$ is a proper stress of the polar tensegrity $(G,p)$. Thus we can determine the forces between beams in a weaving from the forces in the tensegrity. This discussion establishes 
Whiteley's correspondence between the proper stress cones of a tensegrity and its associated weaving:
\begin{theorem}\label{thm: whiteley stress polarity}
The Whiteley stress cone of the weaving $(G,p)$ is linearly 
isomorphic to the proper stress cone of its 
polar tensegrity under the positive diagonal scaling 
$\omega_{ij}\mapsto \omega_{ij}/|s_{ij}|$.
\end{theorem}
As an immediate corollary, we obtain Whiteley's   theorem relating a weaving to its associated tensegrity:
\begin{theorem}\label{cor: flat weaving tensegrity inf eqv}
A weaving is flat if and only if its polar tensegrity  is infinitesimally rigid.
\end{theorem}
Because infinitesimal rigidity of a tensegrity is an open 
property, it is stable with respect to perturbation. Because any sufficiently 
nearby configuration of lines in $\mathbb{R}^3$ with the same over / under pattern 
projects to a nearby weaving, we obtain, via the polarity:
\begin{theorem}\label{thm: locally stable generic}
If $(G,p)$ is a flat weaving, then any sufficiently nearby configuration
of lines in $\mathbb{R}^3$ with the same over / under pattern must be flat.
\end{theorem}

\subsection{Statics correspondence with tensegrities}\label{sec: static polarity}
We can also describe the statics of a weaving.  Denote by $\hat{x}$ the unit vector 
in direction $x$.
The load space on a beam $L_i$ in a weaving is most 
naturally described by a force-couple pair 
$(M_i\hat{p}_i, T_i e_3)$, of a rotational moment 
around an axis through the origin orthogonal to $L_i$ and a 
vertical translational force.  To represent the couple as a 
single vector, we fold the vertical force into its tangential 
component about the origin, yielding a load 
$f_i = M_i\hat p_i - (T_i/\|p_i\|)\widehat{\rho(p_i)}$.%
%
%

The condition for a loading $f$ of a weaving to be 
in equilibrium is then
\[
    \sum_{i\in V(G)} f_i = 0\qquad \text{and}\qquad
    \sum_{i\in V(G)} \langle \rho(p_i),f_i\rangle = 0,
\]
which is formally similar to the condition defining an equilibrium 
loading of a tensegrity.  However, the interpretation 
through the polarity exchanges the roles of forces and 
moments.  Noticing that $\langle \rho(p_i),f_i\rangle = -T_i$, 
we get 
\[
    \sum_{i\in V(G)} \langle \rho(p_i),f_i\rangle = 0
    \qquad \Longleftrightarrow \qquad 
    \sum_{i\in V(G)} T_i = 0,
\]
so the condition for no net torque about the vertical axis in the 
tensegrity load corresponds 
to no net vertical force on the weaving. On the other hand, 
the sum of the $f_i$ vanishing is, by construction, a vanishing 
in-plane moment. One can understand this as saying that 
the moment of tensegrity has a vertical axis and the forces 
are in-plane, whereas for a weaving the moments have in-plane 
axes and the forces are vertical.

A Whiteley stress $\omega$ resolves an equilibrium 
load $f$, if, for each $i$, 
\[
    f_i = \frac{1}{\|p_i\|}\sum_{j: ij\in E} \omega_{ij}\varepsilon_{ij}\rho(q_{ij}),
\]
and we say that the weaving is \emph{statically flat} if every 
equilibrium load on the weaving is resolved by a proper Whiteley stress. 
To expose the polarity from this perspective, recall that 
$\rho(q_{ij}) = -s_{ij}(p_i - p_j)$, so we get 
\[
    \|p_i\|f_i = \sum_{j: ij\in E} -s_{ij}\omega_{ij}\varepsilon_{ij}(p_i - p_j) = 
    -\sum_{j: ij\in E}\omega'_{ij}\varepsilon'_{ij}(p_i - p_j),
\]
and so the loading $f$ on the weaving is resolved by the proper Whiteley stress $\omega$ in the 
weaving if and only if the loading $f'$ described by the diagonal scaling $f_i\mapsto \|p_i\|f_i$
is resolved by the proper stress $\omega'$ in the polar tensegrity.  
We have now exposed another manifestation of the polarity, which, via the 
equivalence of static and infinitesimal rigidity in tensegrities, yields a
new derivation of \Cref{cor: flat weaving tensegrity inf eqv}.
\begin{theorem}\label{thm: static weaving tensegrity polarity}
A weaving is statically flat if and only if its polar tensegrity is 
statically rigid.
\end{theorem}

\subsection{The index theorem}
The polarity between weavings and tensegrities restricts to 
grillages, and 
so the Maxwell--Calladine Index Theorem implies the corresponding statement for 
grillages.
\begin{theorem}[Grillage Index Theorem]\label{thm: index theorem for weavings}
Let $G$ be a graph with $n$ vertices and $m$ edges, and $(G,p)$ 
define a grillage.  If we 
denote by $\ell$ the dimension of the space of non-trivial liftings and 
$s$ the dimension of the space of Whiteley stresses, we have 
\(
    \ell - s + 3 = 2n - m.
\)
\end{theorem}
The Grillage Index Theorem implies that, to design a flat weaving, we must have $m\ge 2n - 2$, 
since \Cref{thm: whiteley flat weaving} requires that $s\ge 1$.  Since the r.h.s. $2n - m$
is the same in the Maxwell--Calladine Index Theorem, we recover a result of Tarnai \cite{Tarnai}:
\begin{theorem}\label{thm: tarnai m - s}
Let $G$ be a graph with $n$ vertices and $m$ edges, and $(G,p)$ 
define a grillage with an $\ell$-dimensional space of non-trivial liftings 
and an $s$-dimensional space of Whiteley stresses.  Denote by $M$ and $S$
the dimensions of the space of non-trivial infinitesimal motions and equilibrium 
stresses of the polar bar framework, respectively.  Then $\ell - s = M - S$.
\end{theorem}


\subsection{The Airy and Whiteley stress functions}
\begin{figure}[h]
  \centering
  \includegraphics[width=0.80\textwidth, trim=0 40 0 30, clip]{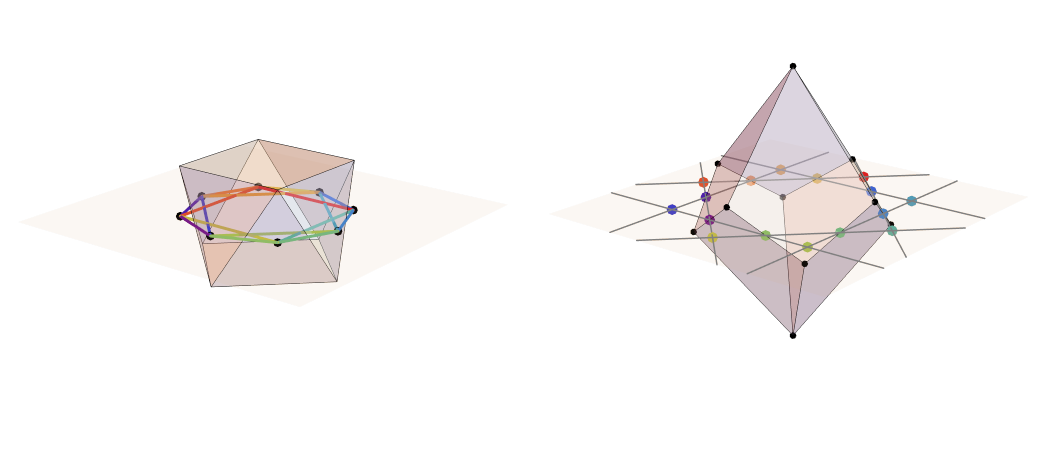}
  \caption{Airy and Whiteley stress functions.  The 
  framework on the left supports an equilibrium stress, as 
  certified by its Airy stress function.  Its polar weaving on the 
  right supports a Whiteley stress, as certified by its Whiteley 
  stress function.
  It is noted that as the tensegrity is the polar dual to the weaving, similarly the Airy stress function is the polar dual to the Whiteley stress function. As such, there is a duality between the stresses in the tensegrity and the Whiteley stresses. }
\end{figure}
In graphic statics, one visualizes an equilibrium stress in a 
tensegrity as a polyhedral lifting with edges projecting vertically
onto the members known as the Airy stress function\footnote{In rigidity, this is 
called the Maxwell--Cremona lifting of a stress.}.  
\begin{definition}\label{def: whiteley stress function}
Let $(G,p)$ be a weaving with a Whiteley stress $\omega$.  
The \emph{Whiteley stress function}
is the polar dual of the Airy stress function of the equilibrium stress $\omega'$ on the polar tensegrity.
\end{definition}
We define the Whiteley stress function this way because a 
weaving does not have a natural facial structure, but the 
polar to the Airy stress function does.  
Although the Airy stress function is not necessarily 
a convex surface, it has a well-defined polar
because $\omega'$ is an equilibrium stress.  We also 
note that, when the origin is not in the interior of
the Airy stress function, the Whiteley 
stress function is non-compact.

\section{Designing a flat weaving}\label{sec:example}
In this section we illustrate how to design a flat weaving using the correspondence developed above. 

\subsection{Methodology}
The input to the method is a pair $(G,p)$, which we interpret as describing a grillage.  The steps are:
\begin{enumerate}
    \item Write the rigidity matrix $R(p)$ of the framework $(G,p)$.
    \item Check that $(G,p)$ is infinitesimally rigid; if it is not, there is no
        stable over / under pattern.
    \item Select an equilibrium stress $\omega$ of $(G,p)$ that is nowhere zero. If none
        exists, throw away all the edges with zero stress for a  
        randomly sampled 
        $\omega$ and go back to the beginning.
    \item For each edge $ij\in E(G)$  set $\varepsilon_{ij} = \operatorname{sign}(\omega_{ij})\operatorname{sign}(s_{ij})$. Thus we have the over / under pattern. 
\end{enumerate}
If the method finishes, the over / under pattern defined by $\varepsilon_{ij}$ determines a flat 
weaving, because $\omega_{ij}/|s_{ij}|$ are the coefficients of a proper Whiteley stress under the 
polarity.

\subsection{Worked example}

\begin{figure}[htp]
\begin{center}
\begin{tikzpicture}[scale=2.3]

\tikzset{
  vtx/.style={circle, draw, fill=black, inner sep=0pt, minimum size=4pt},
  lbl/.style={rectangle, draw=white, fill=white, inner sep=2pt}
}

\node[vtx] (a1) at (75:1.05) {};
\node[vtx] (a2) at (105:1) {};
\node[vtx] (a3) at (255:1) {};
\node[vtx] (a4) at (285:1) {};

\node[vtx] (b1) at (120:1) {};
\node[vtx] (b2) at (150:1) {};
\node[vtx] (b3) at (300:1) {};
\node[vtx] (b4) at (330:1) {};

\foreach \i in {1,...,4}{
  \foreach \j in {1,...,4}{
    \draw (a\i) -- (b\j);
  }
}
\node[rectangle, draw=white, fill=white,
      anchor=south west, xshift=1pt, yshift=1pt] at (a1) {$a_1$};

\node[rectangle, draw=white, fill=white,
      anchor=east, xshift=8pt, yshift=9pt] at (a2) {$a_2$};

\node[rectangle, draw=white, fill=white,
      anchor=north east, xshift=-2pt, yshift=-5pt] at (a3) {$a_3$};

\node[rectangle, draw=white, fill=white,
      anchor=south east, xshift=12pt, yshift=-16pt] at (a4) {$a_4$};

\node[rectangle, draw=white, fill=white,
      anchor=south, yshift=4pt] at (b1) {$b_1$};

\node[rectangle, draw=white, fill=white,
      anchor=south, yshift=6pt] at (b2) {$b_2$};

\node[rectangle, draw=white, fill=white,
      anchor=north, xshift=3pt, yshift=-5pt] at (b3) {$b_3$};

\node[rectangle, draw=white, fill=white,
      anchor=north, yshift=-6pt] at (b4) {$b_4$};
\end{tikzpicture}
\hspace{1cm}
\begin{tikzpicture}[scale=1]
\draw[white] (-2,-2) rectangle (2,2);

\begin{scope}
\clip (-2.7,-2.7) rectangle (2.7,2.7);


\draw[blue, thick, domain=-3:3, samples=2]
plot (\x,{(20/21 - cos(75)*\x)/sin(75)});

\draw[blue, thick, domain=-3:3, samples=2]
  plot (\x,{(1 - cos(105)*\x)/sin(105)});

\draw[blue, thick, domain=-3:3, samples=2]
  plot (\x,{(1 - cos(255)*\x)/sin(255)});

\draw[blue, thick, domain=-3:3, samples=2]
  plot (\x,{(1 - cos(285)*\x)/sin(285)});

\draw[red, thick, domain=-3:3, samples=2]
  plot (\x,{(1 - cos(120)*\x)/sin(120)});

\draw[red, thick, domain=-3:3, samples=2]
  plot (\x,{(1 - cos(150)*\x)/sin(150)});

\draw[red, thick, domain=-3:3, samples=2]
  plot (\x,{(1 - cos(300)*\x)/sin(300)});

\draw[red, thick, domain=-3:3, samples=2]
  plot (\x,{(1 - cos(330)*\x)/sin(330)});






\foreach \A in {75,105,255,285}{
  \foreach \B in {120,150,300,330}{

    \pgfmathsetmacro{\ca}{cos(\A)}
    \pgfmathsetmacro{\sa}{sin(\A)}

    \pgfmathsetmacro{\da}{1}
    \ifnum\A=75
      \pgfmathsetmacro{\da}{20/21}
    \fi

    \pgfmathsetmacro{\cb}{cos(\B)}
    \pgfmathsetmacro{\sb}{sin(\B)}
    \pgfmathsetmacro{\db}{1}

    \pgfmathsetmacro{\det}{\ca*\sb - \sa*\cb}

    \pgfmathsetmacro{\x}{(\da*\sb - \sa*\db)/\det}
    \pgfmathsetmacro{\y}{(\ca*\db - \da*\cb)/\det}

    \fill[black] (\x,\y) circle (1.5pt);
  }
}

\node[blue, fill=white, inner sep=2pt] at (2.2,{(1 - cos(75)*2.2)/sin(75)}) {$a_1$};

\node[blue, fill=white, inner sep=2pt]
      at (1.3,{(1 - cos(105)*2)/sin(105)})
      {$a_2$};

    \node[blue, fill=white, inner sep=2pt]
      at (1.8,{(1 - cos(255)*1.8)/sin(255) - 0.15})
      {$a_3$}; 
\node[blue, fill=white, inner sep=2pt]
      at (2,{(1 - cos(285)*2)/sin(285)})
      {$a_4$};
      \node[red, fill=white, inner sep=2pt]
      at (1.8,{(1 - cos(120)*1.8)/sin(120) + 0.15})
      {$b_1$};
     \node[red, fill=white, inner sep=2pt]
      at (-0.8,0)
      {$b_2$};
     \node[red, fill=white, inner sep=2pt]
      at (-0.9,-1.6)
      {$b_3$};
      \node[red, fill=white, inner sep=2pt]
      at (0.9,0)
      {$b_4$};
\end{scope}

\end{tikzpicture}
\end{center}
\caption{A realization of $K_{4,4}$ in the plane that is infinitesimally rigid, and hence admits a 
$3$-dimensional space of self-stresses. On the right is the associated polar weaving, where each 
line corresponds to a vertex of the framework on the left, and each crossing point (solid dot) 
corresponds to an edge. The over / under structure is determined by the sign of a 
generic self-stress of the framework.
Note that `crossings' only occur where blue and red lines meet. }
\end{figure}
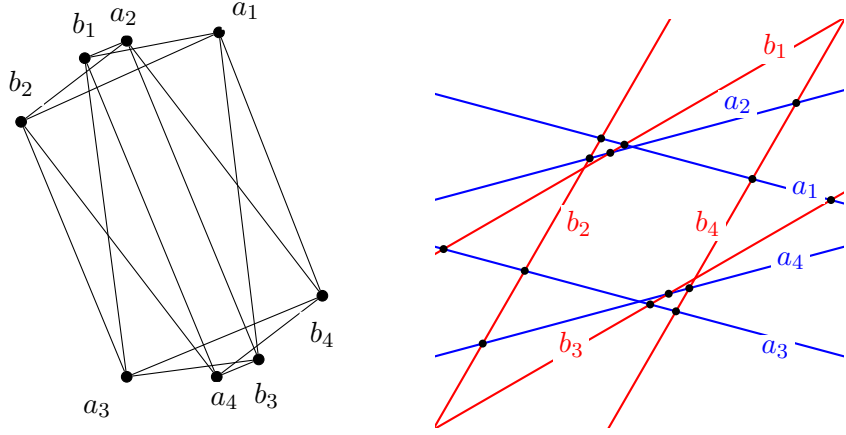


We present an extended example based on the complete bipartite graph $K_{4,4}$.
For convenience, we partition the point configuration 
$p = (a,b)$ into the sub-configurations corresponding to 
the parts of the partition.  We index the edges of 
$K_{4,4}$ by 
$a_1b_1, \ldots, a_1b_4, \ldots, a_4b_1, \ldots, a_4b_4$.

\paragraph*{The configuration}
We use the configuration given by 
$a = (a_1, \ldots, a_4)$ and $b = (b_1, \ldots, b_4)$, where 
\[
a_1=(1 + \frac{1}{20})(\cos \tfrac{5\pi}{12},\sin \tfrac{5\pi}{12}),\;
a_2=(\cos \tfrac{7\pi}{12},\sin \tfrac{7\pi}{12}),\;
a_3=(\cos \tfrac{17\pi}{12},\sin \tfrac{17\pi}{12}),\;
a_4=(\cos \tfrac{19\pi}{12},\sin \tfrac{19\pi}{12})
\]
\[
b_1=(\cos \tfrac{2\pi}{3},\sin \tfrac{2\pi}{3}),\;
b_2=(\cos \tfrac{5\pi}{6},\sin \tfrac{5\pi}{6}),\;
b_3=(\cos \tfrac{5\pi}{3},\sin \tfrac{5\pi}{3}),\;
b_4=(\cos \tfrac{11\pi}{6},\sin \tfrac{11\pi}{6}).
\]
Figure 3 shows both the weaving and associated framework. 

\paragraph*{Steps 1 and 2: Infinitesimal rigidity}
The rigidity matrix $R(a,b)$ is the $16 \times 16$ matrix whose rows are
indexed by the edges $a_i b_j$ and whose columns are indexed by the
vertices $a_1,\ldots,a_4,b_1,\ldots,b_4$ (each contributing two columns
for its $x$- and $y$-coordinates).  The row for edge $a_i b_j$ has the
block $(a_i - b_j)^\top$ in the two columns of $a_i$, the block
$(b_j - a_i)^\top$ in the two columns of $b_j$, and zeros elsewhere:
{\tiny
\[
R(a,b) =
\bordermatrix{
        & a_1 & a_2 & a_3 & a_4 & b_1 & b_2 & b_3 & b_4 \cr
a_1b_1  & a_1-b_1 & 0 & 0 & 0 & b_1-a_1 & 0 & 0 & 0 \cr
a_1b_2  & a_1-b_2 & 0 & 0 & 0 & 0 & b_2-a_1 & 0 & 0 \cr
a_1b_3  & a_1-b_3 & 0 & 0 & 0 & 0 & 0 & b_3-a_1 & 0 \cr
a_1b_4  & a_1-b_4 & 0 & 0 & 0 & 0 & 0 & 0 & b_4-a_1 \cr
a_2b_1  & 0 & a_2-b_1 & 0 & 0 & b_1-a_2 & 0 & 0 & 0 \cr
\vdots  &\vdots  & & & & & & & \vdots \cr
a_4b_4  & 0 & 0 & 0 & a_4-b_4 & 0 & 0 & 0 & b_4-a_4 \cr
}
\] 
}
where each displayed entry $a_i - b_j$ denotes the 
$1\times 2$ row vector $(a_i - b_j)^\top$ occupying the 
two coordinate columns of the indicated vertex.

Infinitesimal rigidity can be checked either by direct computation 
or using the fact 
that a framework $(K_{4,4},a,b)$ in general 
position is infinitesimally flexible if and only if the configuration 
lies on a conic.  In this case, the points $a_2,  \ldots, b_4$
lie on the unit circle, which considering dimensions is the 
only conic through all of them, and $a_1$ is perturbed off the 
unit circle.  Hence, $(K_{4,4},a,b)$ is infinitesimally rigid.

\paragraph*{Step 3: Select an equilibrium stress}
The Maxwell--Calladine Index Theorem implies that $(K_{4,4},a,b)$ has a
$3$-dimensional space of self-stresses.  We can select one 
computationally, but in this case, we appeal to the fact that, if we 
let 
\[
    \alpha_1a_1 + \cdots \alpha_4a_4 = 0\qquad\text{and}\qquad
    \beta_1b_1 + \cdots \beta_4b_4 = 0
\]
be the unique and nowhere zero (by general position) 
affine dependencies among the parts of the configuration, 
then $\omega_{a_ib_j} = \alpha_i\beta_j$ is an equilibrium 
stress that is non-zero on every edge.

\paragraph*{Step 4: Compute the over / under pattern}
The final step is to compute the signs $\varepsilon_{a_ib_j}$.
Working numerically we find
that the affine dependencies are (up to scale)
\[
\alpha \approx (-0.488,\ 0.500,\ -0.512,\ 0.500) \qquad\text{and}
\qquad
\beta = \tfrac{1}{2}(-1,\ 1,\ -1,\ 1).
\]
Combining the sign of $\omega_{a_ib_j}$ with the
sign of $s_{a_ib_j} = \det[a_i\, b_j]$ yields the over / under pattern
$\varepsilon_{a_ib_j} = \operatorname{sign}(\omega_{a_ib_j})\,
\operatorname{sign}(s_{a_ib_j})$:
\[
\varepsilon =
\bordermatrix{
     & b_1 & b_2 & b_3 & b_4 \cr
a_1  & + & - & - & + \cr
a_2  & - & + & + & - \cr
a_3  & - & + & + & - \cr
a_4  & + & - & - & + \cr
}
\]
Moreover, the magnitude of the forces between the beams at the 
contact point $q_{ij}$ is, up to a global positive scaling, 
$|\alpha_i\beta_j|\det [a_i\, b_j]^{-2}$.  
Thus, we remarkably obtain all the information required to understand the behaviour of weaving whilst only considering the dual tensegrity, and we have designed an over / under pattern which is stable.
It is easy to see how one could use the underlying tensegrity framework as a design freedom to develop well-behaved weavings. 


\section{Conclusions and outlook}
We provided an updated exposition of the investigations of Whiteley and Tarnai into the 
stability of planar structures comprising of woven elastic beams, completing the theory to 
include statics, and deriving a principled method for computing stable over / under 
patterns for a fixed geometry.  A remaining combinatorial question is to enumerate
\emph{all} stable over / under patterns.  This may be difficult, since it is closely 
related to the question of enumerating possible orientations of a linear matroid, but 
may be approachable for small problem sizes.

An intriguing question is to find conditions in which woven structures can be 
structurally rigid without 
relying on  frictional (no slip) constraints between beams, as opposed to simply remaining flat.  
This is impossible in the idealized flat case, considered here, because
as discussed in \S \ref{sec: static polarity}, a flat weaving can only 
respond to a subspace of all $3$D equilibrium loads.  Weavings on a 
curved surface are not so restricted, opening the possibility of a 
full rigidity theory.

\section*{Acknowledgements} BS was partially supported by UK Research and 
Innovation through the Small Grant Scheme project UKRI2397.
LST was partially supported by UK Research and Innovation through the 
Small Grant Scheme project UKRI1112 and St Andrews Impact and Innovation.
CM, BS and LST thank ICERM, which is supported by the National Science Foundation (NSF)
under Grant No. DMS-2424556, for its hospitality during the semester program 
``Geometry of Packings, Materials, and Rigid Frameworks.''


\printbibliography


\end{document}